\documentclass[12pt]{amsart}
\usepackage{tikz}
\usetikzlibrary{calc, arrows, decorations.markings} \tikzset{>=latex}
\usepackage{amsmath,amssymb,amsfonts}
\usepackage{amscd, amssymb, latexsym, amsmath, amscd}
\usepackage{amsrefs}
\usepackage[all]{xy}
\usepackage{pb-diagram}
\usepackage{enumerate}
\usepackage{anysize}
\usepackage[sc]{mathpazo} 
\usepackage{enumitem}
\marginsize{2.5cm}{2.5cm}{2.5cm}{2.5cm}
\usepackage{graphics}
\usepackage{color}
\DeclareFontFamily{U}{wncy}{}
    \DeclareFontShape{U}{wncy}{m}{n}{<->wncyr10}{}
    \DeclareSymbolFont{mcy}{U}{wncy}{m}{n}
    \DeclareMathSymbol{\Sh}{\mathord}{mcy}{"58}

\input xy
\xyoption{all}
\usepackage{pb-diagram}
\usepackage[all]{xy}
\input xy
\xyoption{all}
\theoremstyle{plain}
\newtheorem{thm}{Theorem}[section]
\newtheorem{theorem}{Theorem}[section]

\newtheorem{lemma}{Lemma}[section]
\newtheorem{cor}[theorem]{Corollary}

\newtheorem{prop}{Proposition}[section]

\newtheorem{conj}{Conjecture}[section]

\theoremstyle{definition} \theoremstyle{definition}
\newtheorem{remark}{Remark}[section]
\newtheorem{question}{Question}[section]
\newtheorem{example}{Example}[section]

\theoremstyle{remark}

\newcommand{\Sc}{\mathcal{S}}

\newcommand{\PP}{\mathcal{P}}

\newcommand{\G}{\textsc{\G}}

\newcommand{\U}{\mathcal{U}}

\newcommand{\Z}{\mathbb{Z}}

\newcommand{\C}{\mathbb{C}}

\newcommand{\F}{\mathfrak{F}}

\newcommand{\Wh}{{\rm Wh}}

\newcommand{\RR}{\mathcal{R}}

\newcommand{\Tor}{{\rm Tor}}
\newcommand{\Hom}{{\rm Hom}}
\newcommand{\EP}{{\rm EP}}
\newcommand{\Bes}{{\rm Bes}}

\newcommand{\Ext}{{\rm Ext}}
\newcommand{\Ind}{{\rm Ind}}
\newcommand{\ind}{{\rm ind}}

\def\G{{\rm G}}

\def\St{{\rm St}}

\def\U{{\rm U}}

\def\GL{{\rm GL}}
\def\PGL{{\rm PGL}}

\def\SO{{\rm SO}}
\def\OO{{\rm O}}

\def\O{{\mathcal O}}

\begin{document}

\title{Ext-analogues of Branching laws}

\author{Dipendra Prasad}
\address{
Tata Institute of Fundamental Research \\
Homi Bhabha Road \\
Bombay - 400 005 \\ \
INDIA}
\email{dprasad@math.tifr.res.in}

\subjclass[2010]{Primary 11F70; Secondary 22E55}

\date{\today}

\begin{abstract}
  We consider the Ext-analogues of branching laws for representations of a group to its subgroups in the context of $p$-adic groups.

  \end{abstract}

\maketitle

\setcounter{tocdepth}{1}

\tableofcontents

\section{Introduction}

Considering the restriction of representations of a group $G$ to one of its subgroups $H$, say of $G=\SO_{n+1}(F)$ to $H= \SO_n(F)$ 
for  a non-archimedean local field $F$  has been a very fruitful 
direction of research especially through its connections to questions on period integrals of automorphic representations, cf. \cite{ggp} for the conjectural theory both locally and globally.
The question for local fields amounts to understanding
$\Hom_{\SO_n(F)}[\pi_1,\pi_2]$ 
for irreducible admissible representations $\pi_1$ of $\SO_{n+1}(F)$, and $\pi_2$ of $\SO_n(F)$.
The first result proved about this
is the multiplicity one property
which says that this space  is at most one dimensional, cf. \cite{multiplicity},  \cite{sun-zhu}. It may be mentioned that before  
the full multiplicity one theorem was proved, even finite dimensionality  of the space was not 
known. With multiplicity one theorem proved,
one then goes on to prove more precise
description of the set of irreducible admissible representations $\pi_1$ of $\SO_{n+1}(F)$, and $\pi_2$ of $\SO_n(F)$
with  $\Hom_{\SO_n(F)}[\pi_1,\pi_2] \not = 0$. These have now become available in a series of papers due to 
Waldspurger, and Moeglin-Waldspurger, cf. \cite{wald:gp}, \cite{wald:gp1}, \cite{wald:gp2}, \cite{wald:gp3}. 
There is also a recent series of papers by 
Beuzart-Plessis on similar questions for unitary groups, cf. \cite{raphael1},  \cite{raphael2},  \cite{raphael3}.

Given the interest in the space $\Hom_{\SO_n(F)}[\pi_1,\pi_2]$, 
it is natural
to consider the related spaces $\Ext^i_{\SO_n(F)}[\pi_1,\pi_2]$, and in fact 
homological algebra methods suggest that the simplest answers are not for
these individual spaces, but for the 
alternating sum of their dimensions:
$\EP[\pi_1,\pi_2] = \sum_{i=0}^{\infty}(-1)^i\dim \Ext^i_{\SO_n(F)}[\pi_1,\pi_2];$
these hopefully more manageable 
objects --certainly more flexible-- when coupled with vanishing of higher $\Ext$'s (when available) 
may give theorems about 
$\Hom_{\SO_n(F)}[\pi_1,\pi_2]$. We hasten to add that before we can define $\EP[\pi_1,\pi_2]$, 
$\Ext^i_{\SO_n(F)}[\pi_1,\pi_2]$ 
needs to  be proved to be finite dimensional 
for $\pi_1$ and $\pi_2$ finite length admissible representations 
of $\SO_{n+1}(F)$ and $\SO_n(F)$ respectively, and also
proved to be 0 for $i$ large.  Vanishing of $\Ext^i_{\SO_n(F)}[\pi_1,\pi_2]$ for large $i$ is a well-known generality to which we will come to later.
Towards a proof  of finite dimensionality 
of $\Ext^i$ in this case, to be made by an inductive argument
on $n$ later in the paper,  we  note 
that unlike $\Hom_{\SO_n(F)}[\pi_1,\pi_2]$, where we will have no idea how to
prove finite dimensionality if both  $\pi_1$ and $\pi_2$ are cuspidal, exactly
this case we can handle apriori, for $i> 0$, as almost by the very  definition  of cuspidal
representations, they are  both projective and injective objects in the category of smooth representations. Recently, there is a very general 
finiteness theorem for $\Ext^i[\pi_1,\pi_2]$ (for spherical varieties) by Aizenbud and Sayag in \cite{sayag}. However, we have preferred to give our own older approach via Bessel models which intervene when analyzing principal series representations of $\SO_{n+1}(F)$ when
restricted to $\SO_n(F)$. As a bonus, this approach gives explicit answers about
Euler-Poincar\'e characteristics.

Thinking about $\Ext$-analogues suggest
interchanging the roles of $\pi_1$ and $\pi_2$ in analogy with the known relationship,
$\EP[V_1,V^\vee_2] = \EP[V_2,V^\vee_1]$ when $V_1$ and $V_2$ are finite length representations on the same group, and allows one to consider submodules as in 
$\Hom_{\SO_n(F)}[\pi_2,\pi_1]$, and more generally,
$\Ext^i_{\SO_n(F)}[\pi_2,\pi_1]$.

Based on various examples, a clear picture seems to be emerging
about $\Ext^i_{H}[\pi_1,\pi_2]$.
For example, we expect that when $\pi_1$ and $\pi_2$ are tempered,
$\Ext^i_H[\pi_1,\pi_2]$  is nonzero only for $i=0$.
On the other hand we expect that $\Ext^i_H[\pi_2,\pi_1]$ is typically zero for $i=0$ (so no wonder branching is usually not considered as a subrepresentation!),
and shows up only for  $i$ equals the split rank of the center of the Levi 
from which $\pi_2$  arises through parabolic induction of a 
supercuspidal representation; in fact $\Ext^i_H[\pi_2,\pi_1]$ is zero beyond the split rank of the center of this Levi by 
generalities, so
$\Ext^i_H[\pi_1,\pi_2]$ is typically nonzero only for $i=0$, whereas 
 $\Ext^i_H[\pi_2,\pi_1]$ is nonzero only for the largest possible $i$. 
We make precise  some of these suggestions during the course of the paper, and discuss some examples as evidence for the suggested conjectures made here.

 In the process of relating $\Ext^i[\pi_1,\pi_2]$ with $\Ext^i[\pi_2,\pi_1]$, we were led to a 
duality theorem for a general reductive group which turned out to be a consequence of the work of 
Schneider and Stuhler in \cite{schneider-stuhler:sheaves}.
It is the subject matter of section \ref{section10}. As an example,
$\Hom_{\PGL_2(F)}[ \pi_1 \otimes \pi_2, \pi_3]$
which was part of author's work
in \cite {prasad1}, and simple calculations about
$\EP_{\PGL_2(F)}[ \pi_1 \otimes \pi_2, \pi_3]$ allow the calculation of 
$\Ext^1_{\PGL_2(F)}[ \pi_1 \otimes \pi_2, \pi_3]$, and then by the
duality theorem, we  are able to
analyze $\Hom_{\PGL_2(F)}[ \pi_3, \pi_1 \otimes \pi_2]$ (irreducible submodules of
the tensor product).

In the archimedean case, several papers of T. Kobayashi, see e.g. \cite{kobayashi},  do study the restriction
problem for $(\mathfrak g,K)$-modules in the sense of sub-modules but the analogous restriction problem in the sense of sub-modules seems to be absent in the $p$-adic case.
Propositions \ref{nosub} and \ref{submodule} suggest that $\Hom_H[\pi_2,\pi_1] = 0$
whenever $\pi_1$ is an irreducible tempered representation of $G$
(assumed to be simple) unless $H$ has compact center, and $\pi_2$ is a supercuspidal representation of it.

To summarize the main results of the paper,
we might mention Theorem
\ref{whittaker} giving a complete understanding of $\EP_{\GL_n(F)}[\pi_1,\pi_2]$ for $\pi_1$ and $\pi_2$ finite length representations of $\GL_{n+1}(F)$ and $\GL_n(F)$ respectively. Theorem \ref{finite} proves $\Ext^i_{\SO_n(F)}[\pi_1,\pi_2]$
to be  finite dimensional for $\pi_1$ and $\pi_2$ finite length representations of $\SO_{n+1}(F)$ and $\SO_n(F)$ respectively, and as 
Corollary 
\ref{epson} of the proof,  gives a good understanding of $\EP_{\SO_n(F)}[\pi_1,\pi_2]$ when $\pi_1$ is a principal series. 
We formulate as  Conjecture \ref{vanishing} the  vanishing of 
$\Ext^i_{\GL_n(F)}[\pi_1,\pi_2]$ for $i>0$ for generic representations, 
and Conjecture \ref{integral} suggests that the
integral formula discovered by Waldspurger in  \cite{wald:gp} and
\cite{wald:gp1}  are actually for
Euler-Poincar\'e characteristic  of general finite length representations
in the spirit of Kazhdan orthogonality. In section \ref{geometrization} we
suggest that all nontrivial Ext's have some `geometric' origin.

\vspace{4mm}

\noindent{\bf Acknowledgment: } The author thanks Ann-Marie Aubert, Jeffrey Adams, U.K. Anandavardhanan, Joseph Bernstein, Atsushi Ichino, 
Raphael Beuzart-Plessis, Peter Schneider, Maarten Solleveld, and Sandeep Varma  for useful discussions and correspondence. This paper is an updated version of the paper with the same title in arXiv from 2013!

\section{Preliminaries}
Given a connected reductive $F$-group $G$, we make the usual abuse of  notation to also
denote by $G$ the locally compact totally disconnected group $G(F)$ of $F$-rational points of the algebraic group 
$G$. We denote by $\RR(G)$ the abelian category of smooth representations of $G$ over $\C$.
The abelian category $\RR(G)$ has enough projectives and enough injectives,
e.g. for any compact open
subgroup $K$ of $G$, $\ind_K^G(\C)$ is a projective object in $\RR(G)$, and $\Ind_K^G(\C)$ is an injective object in $\RR(G)$  (we use throughout the paper $\ind$ for compactly supported induction
and $\Ind$ for induction without compact support condition); in fact these projective objects and their direct summands, and their smooth duals as injective objects suffice for all considerations in the paper.
Since $\RR(G)$ has enough projectives and enough injectives, it is meaningful to talk about
$\Ext^i_G[\pi_1,\pi_2]$ as the derived functors of
$\Hom_G[\pi_1,\pi_2]$.

For reductive $p$-adic groups $G$ considered in this paper, it is 
known that $\Ext^i_G[\pi,\pi']$
is zero for any two smooth representations $\pi$ and $\pi'$ of $G$ when $i$ is greater than the $F$-split rank of $G$.
This is a standard application of the projective resolution of the trivial representation
$\C$ of $G$ provided by the building associated to $G$. For another proof of this,
and for 
finite dimensionality
of $\Ext^i_G[\pi,\pi']$, see Proposition \ref{ext-ps} below.

For two smooth
representations $\pi$ and $\pi'$
of $G$ one can consider
the {Euler-Poincar\'e pairing} $\EP_G[\pi,\pi']$
between $\pi$ and $\pi'$  defined by
$$
\EP_G[\pi,\pi'] = \sum _i (-1)^i\dim_\C \Ext^i_G[\pi,\pi'].
$$
For this definition to make sense, we must first prove that  $\Ext^i_G[\pi,\pi']$ is a finite-dimensional vector space over $\C$ for all integers $i$.
An obvious remark which will be tacitly used throughout this paper is that if 
$$0 \rightarrow \pi_1 \rightarrow \pi \rightarrow \pi_2 \rightarrow 0,$$
is an exact sequence of smooth $G$-modules, then if any two of the $\EP_G[\pi_1,\pi'], \EP_G[\pi,\pi'],$  $\EP_G[\pi_2,\pi'], $ make sense,
then so does the third (finite dimensionality of the Ext groups, and zero beyond a stage), and 
$$\EP_G[\pi,\pi'] =  \EP_G[\pi_1,\pi'] +  \EP_G[\pi_2,\pi']. $$
This remark will be used
to break up representations $\pi$ or $\pi'$ in terms of simpler objects for which $\EP$ can be proved to make sense
by reducing to smaller groups via some form of Frobenius recoprocity.

The 
following proposition summarizes some key properties of the Euler-Poincar\'e pairing, see \cite{schneider-stuhler:sheaves} for the proofs (part \ref{prop:EP} (\ref{item:EP-chars}) is known only in characteristic zero).

\begin{prop}
\label{prop:EP}
Let $\pi$ and $\pi'$ be finite-length, smooth
representations of a
reductive
$p$-adic
group $G$.
Then:
\begin{enumerate}
\item
\label{item:EP-bilinear}
$\EP[\pi^\vee_1,\pi_2]$ is a symmetric, $\Z$-bilinear form
on the Grothendieck group of finite-length
representations of $G$.
\item
\label{item:EP-loc-const}
$\EP$ is locally constant.
(A family $\{\pi_\lambda\}$ of representations on a fixed vector space $V$ 
is said to \emph{vary continuously}
if all $\pi_\lambda|_K$ are all equivalent for some compact open subgroup $K$,
and the matrix coefficients
$\langle \pi_\lambda v, \tilde v\rangle$
vary continuously in $\lambda$.)

\item
\label{item:EP-ps}
$\EP_G[\pi,\pi'] = 0$
if $\pi$ or $\pi'$ is
induced from any proper parabolic subgroup in $G$.

\item
\label{item:EP-chars}
$\EP_G[\pi,\pi'] = 
\int_{C_{ellip}} \Theta(c)\bar{\Theta}'(c)\, dc$,
where $\Theta$ and $\Theta'$ are the characters of $\pi$ and $\pi'$ assumed to have the same
unitary central character, 
and $dc$ is a natural measure on the set
$C_{ellip}$ of regular elliptic conjugacy classes in $G$. (Note that if $G$ has non-compact center, then both sides
of this equality are zero; the right hand side being zero as there are no regular elliptic elements in $G$ in that 
case, and  the left hand side being zero by a simple argument.)
\end{enumerate}
\end{prop}

Several assertions about $\Hom$ spaces can be  converted into assertions about $\Ext^i$.
The following generality allows one to do so.

\begin{prop}\label{adjoints}
  Let ${\mathcal A}$ and ${\mathcal B}$ be two abelian categories, and ${\mathcal F}$ a functor from 
${\mathcal A}$ to ${\mathcal B}$, and ${\mathcal G}$ a functor from 
${\mathcal B}$ and ${\mathcal A}$. Assume that  ${\mathcal G}$ is a left adjoint of  ${\mathcal F}$, i.e., 
there is a natural equivalence of functors:
$$\Hom_{\mathcal B}[X, {\mathcal F}(Y)] \cong \Hom_{\mathcal A}[{\mathcal G}(X), Y].$$
Then,
\begin{enumerate}
\item If ${\mathcal F}$ and ${\mathcal G}$ are exact functors, then
  ${\mathcal F}$ maps injective objects of ${\mathcal A}$ to injective
objects of ${\mathcal B}$, and ${\mathcal G}$ maps projective objects of ${\mathcal B}$ to projective
objects of ${\mathcal A}$.

\item If ${\mathcal F}$ and ${\mathcal G}$ are exact functors, then
  $\Ext^i_{\mathcal B}[X, {\mathcal F}(Y)] \cong \Ext^i_{\mathcal A}[{\mathcal G}(X), Y].$

\end{enumerate}
\end{prop}
\begin{proof} Part $(1)$ of the Proposition follows directly from definitions; see Bernstein \cite{ber}, Proposition 8. For part $(2)$, it suffices to note that 
if $$\cdots P_n \rightarrow P_{n-1} \rightarrow \cdots \rightarrow P_1 \rightarrow P_0 \rightarrow X \rightarrow 0,$$
is a projective resolution of an object $X$ in ${\mathcal B}$, then by part $(1)$ of the Proposition,
 $$\cdots {\mathcal G}(P_n) \rightarrow {\mathcal G}(P_{n-1}) \rightarrow  \cdots \rightarrow {\mathcal G}(P_1) \rightarrow {\mathcal G}(P_0) \rightarrow {\mathcal G}(X )\rightarrow 0,$$
is a projective resolution of ${\mathcal G}(X)$. Therefore part $(2)$ of the proposition
follows from the adjointness relationship between ${\mathcal F}$ and ${\mathcal G}$.
\end{proof}

The following is a direct consequence of  Frobenius reciprocity combined with 
Proposition \ref{adjoints}.

\begin{prop} \label{closed} Let  $H$ be a closed subgroup of a $p$-adic Lie group $G$. Then, 
\begin{enumerate} 
\item The restriction
of any smooth projective representation of $G$ to $H$
is a projective object in $\RR(H)$, and ${\Ind}_H^G U$ is an injective representation 
of $G$ for any  injective representation $U$ of $H$.

\item For any 
smooth representation $U$ of $H$, and $V$ of $G$,
$$\Ext^i_G[V, {\rm Ind}_H^GU] \cong \Ext^i_H[V,U].$$
\end{enumerate}
\end{prop}

Note that 
for any two smooth representations $U,V$ of $G$, 
$$\Hom_G[U, V^\vee] \cong \Hom_G[{V},{U}^\vee],$$
where ${U}^\vee, {V}^\vee$ are  the smooth duals
of $U,V$ respectively.  Therefore we have adjoint functors as in 
Proposition \ref{adjoints} with ${\mathcal F}= {\mathcal G}$ to be the smooth dual 
from the category of smooth representations of a $p$-adic group $G$ 
to its opposite category. By Proposition \ref{adjoints},
it follows that the smooth dual of a projective object
in $\RR(G)$ is an injective object in $\RR(G)$, and further we have the following proposition.

\begin{prop} For  a $p$-adic Lie group $G$, let $U$ and $V$ be two smooth representations of $G$. 
Then,
$$\Ext^i_G[U, {V}^\vee] \cong \Ext^i_G[{V},{U}^\vee].$$
\end{prop}

Since the smooth dual of ${\rm ind}_H^G (U)$ is ${\rm Ind}_H^G({U}^\vee)$ (for normalized induction),
the previous two 
propositions combine to give:

\begin{prop} \label{frobc} For $H$ a closed subgroup of a $p$-adic Lie group $G$, let $U$ be a 
smooth representation of $H$, and $V$ a smooth  representation of $G$. Then,
$$  \Ext^i_H[{V},{U}^\vee] \cong  \Ext^i_G[{V}, {\rm Ind}_H^G({U}^\vee)]  \cong 
\Ext^i_G[{\rm ind}_H^GU, {V}^\vee] .
$$
\end{prop}

For smooth representations $U,V,W$ of $G$,
the  canonical isomorphism,
$$\Hom_G[{V \otimes U},{W}^\vee] \cong \Hom_G[{U \otimes W},{V}^\vee],$$
translates into 
 the following proposition by Proposition \ref{adjoints}.

\begin{prop}
For  a $p$-adic Lie group $G$, and  $U, V,W$  smooth representations of $G$,
 there are  canonical isomorphisms,
$$\Ext^i_G[{V \otimes U},{W}^\vee] \cong \Ext^i_G[{U \otimes W},{V}^\vee],$$
in particular $$\Ext^i_G[V \otimes U, \C] \cong \Ext^i_G[{V},{U}^\vee].$$
\end{prop}

Proposition \ref{adjoints} with the form of Frobenius reciprocity for Jacquet modules, implies the following 
proposition.

\begin{prop} \label{frob} For $P$ a parabolic  subgroup of a reductive $p$-adic group $G$ with Levi decomposition $P=MN$, the Jacquet functor
  $V\rightarrow V_N$ from $\RR(G)$ to $\RR(M)$ takes projective objects to projective objects, and
  for $V \in \RR(G)$, $U\in \RR(M)$, we have
  (using normalized parabolic induction and normalized Jacquet module),
  $$\Ext^i_G[V, {\rm Ind}_P^GU] \cong \Ext^i_M [V_N,U].$$
\end{prop}

The proof of the following proposition is exactly as the proof of the earlier proposition. This proposition will play an important role
in setting-up an inductive context to prove theorems on a group $G$ in terms of similar theorems for subgroups.

\begin{prop}\label{bessel} For $P$ a  (not necessarily parabolic) subgroup of a reductive $p$-adic group $G$ with Levi decomposition $P=MN$, 
let $\psi$ be a character of $N$ 
normalized by  $M$. Then for any irreducible representation $\mu$ of $M$, one can define a representation of $P$, denoted by 
$\mu \cdot \psi$ which when restricted to $M$ is $\mu$, and when restricted to $N$ is $\psi$.  
For any smooth representation $V$ of $G$, let $V_{N,\psi}$ be the twisted Jacquet module of $V$ with respect to the
character $\psi$ of $N$ which is a smooth representation of $M$. Then,
$$\Ext^i_G[ {\rm ind}_P^G (\mu \cdot \psi), {V}^\vee] \cong \Ext^i_M [ V_{N, \psi}, {\mu}^\vee].$$
\end{prop}

The following much deeper result than these earlier results follows from  the so called {\it Bernstein's second adjointness theorem}.

\begin{thm}\label{second-ad}
 For $P$ a parabolic  subgroup of a reductive $p$-adic group $G$ with Levi decomposition $P=MN$, let $U$ be a 
smooth representation of $M$ thought of as a representation of $P$, and $V$ a smooth representation of $G$.
Let $P^- = MN^-$ be the parabolic opposite to $P=MN$. Then we have (using
normalized parabolic induction and normalized Jacquet module),
$$\Ext^i_G[ {\rm Ind}_P^GU, V] \cong \Ext^i_M [U, V_{N^-}].$$
\end{thm}

As a sample of arguments with the Ext groups, we give a proof of the following basic proposition. 

\begin{prop}
\label{ext-ps}
Suppose that $V$ is a smooth representation of $G$ of  finite length, and that all of its
irreducible subquotients are subquotients of representations induced from supercuspidal representations
of a Levi factor of the standard parabolic subgroup $P=MU$ of $G$,
defined by a subset $\Theta$ of the set of simple roots for a maximal split torus of $G$.
Then if $V'$ is a finite length smooth representation of $G$,
$\Ext^i_{G}[V,V']$
and $ \Ext^i_{G}[V',V] $ are finite dimensional vector spaces over $\C$.
If $V'$ is any smooth representation of $G$,  
$\Ext^i_{G}[V,V'] = \Ext^i_{G}[V',V] = 0$
for $i> d(M)= d - |\Theta|$ where $d$ is the $F$-split rank of $G$.
Further, $\EP_G[\pi,\pi'] = 0$
if both $\pi,\pi'$ are finite length representations of $G$, and $\pi$ or $\pi'$ is
induced from a proper parabolic subgroup of $G$.
\end{prop}

\begin{proof}
  This is \cite{schneider-stuhler:sheaves}*{Corollary III.3.3}.
  Since this is elementary enough, we give another proof.

  We begin by noting that tensoring $V$ by the resolution of $\C$ by
  projective modules in $\RR(G)$ afforded by the building associated to $G$
  gives a projective resolution of $V$, but one which is not finitely generated
  as a $G$-module even if $V$ is irreducible, and therefore proving finite dimensionality of
  $\Ext^i_{G}[V,V']$ requires some work. The resolution given by the building
  at least proves that these are 0 beyond the split rank of $G$. Our proof below first proves the assertions on
  $ \Ext^i_{G}[V,V'] $ if $V$ or $V'$ is a full principal series
  ${\rm Ind}_P^G\rho$ where $\rho$ is a cuspidal representation of $M$, and then
    handles all subquotients by a standard {\it d\'evissage}.

    Fix a surjective map $\phi: M \rightarrow \Z^{d(M)}$ with kernel $M^\phi$ which is
    sometimes called the subgroup of $M$ generated by compact elements.
    
    Let $\rho$ be a cuspidal representation of $M$.
    Therefore $\rho$ restricted to
    $M^\phi$,  which is $[M,M]$ up to a compact group,  is an injective module,
    and hence $\Ext^i_{M^\phi} [V_N,\rho] = 0$ for $i>0$.
        By Frobenius reciprocity, combined with the spectral sequence
        associated to the normal subgroup $M^\phi$ of $M$ with quotient $\Z^d$,
        it follows that:
\begin{eqnarray*} \Ext^i_G[V, {\rm Ind}_P^G\rho]
  & \cong & \Ext^i_M [V_N,\rho] \\
  & \cong &  H^i(\Z^{d(M)}, \Hom_{M^\phi} [V_N,\rho]).
\end{eqnarray*}

This proves that $\Ext^i_G[V, {\rm Ind}_P^G\rho]=0$
for $i>d(M)$ for any smooth representation $V$ of $G$, and
that $\Ext^i_G[V, {\rm Ind}_P^G\rho]$ are finite dimensional for $V$ of finite length.

Similarly, by the second adjointness theorem, it follows that
$\Ext^i_G[{\rm Ind}_P^G(\rho) ,V']
=0$ for $i>d(M)$, and that  
  $\Ext^i_G[{\rm Ind}_P^G(\rho) ,V']$ are
   finite dimensional for $V'$ a finite length representation in $\RR(G)$.

  Having proved properties of $\Ext^i_G[V, {\rm Ind}_P^G\rho]$ and
  $\Ext^i_G[{\rm Ind}_P^G(\rho) ,V']$,
  the rest of the proposition about $\Ext^i_G[V,V']$
  follows by {\it d\'evissage} by writing an irreducible representation $V$ of $G$
  as a quotient of a principal series $Ps ={\rm Ind}_P^G\rho$, and using conclusions on the principal series to make conclusions on $V$. This part of the argument is very similar to what we give in Lemma \ref{6.1}, so we omit it here. \end{proof}

\section{Kunneth Theorem}

In this section, we prove a form of the  Kunneth theorem which we will have several occasions to use. 
A version of Kunneth's theorem is there in \cite {raghuram} assuming, however, finite length conditions on both
$E_1$ and $E_2$ which is not adequate for our applications.

During the course of the proof of the Kunneth Theorem, we will need to use the following 
most primitive form of  Frobenius reciprocity.

\begin{lemma}Let $K$ be an open subgroup of a $p$-adic group $G$. Let $E$ be a smooth representation of 
$K$, and $F$ a smooth representation of $G$. Then,
$$\Hom_G[ {\rm ind}_{K}^GE, F] \cong \Hom_{K} [E, F].$$
\end{lemma}

\begin{thm} \label{kunneth} Let $G_1$ and $G_2$ be two $p$-adic groups. Let $E_1, F_1$ be any two 
smooth representations of $G_1$, and $E_2,F_2$ be any two smooth representations
of $G_2$. Then assuming that $G_1$ is a reductive $p$-adic group, and 
$E_1$ has finite length, we have
$$\Ext^i_{G_1 \times G_2}[E_1\boxtimes E_2, F_1\boxtimes F_2] \cong \bigoplus_{i=j+k} \Ext^j_{G_1}[E_1,F_1] \otimes \Ext^k_{G_2}[E_2,F_2].$$
\end{thm}

\begin{proof} If $P_1$ is a projective module for $G_1$, and $P_2$ a projective module for $G_2$, then
  $P_1 \boxtimes P_2$ is a projective module for $G_1 \times G_2$.

Let
$$\cdots \rightarrow P_1 \rightarrow P_0 \rightarrow E_1 \rightarrow 0,$$ 
$$\cdots \rightarrow Q_1 \rightarrow Q_0 \rightarrow E_2 \rightarrow 0,$$
be a projective resolution for $E_1$ as a $G_1$-module, and  a projective resolution
for $E_2$ as a $G_2$-module.

It follows that the tensor product of these two exact sequences: 
$$\cdots \rightarrow P_1\boxtimes Q_0 + P_0 \boxtimes Q_1  \rightarrow P_0\boxtimes Q_0  \rightarrow E_1 \boxtimes E_2 
\rightarrow 0,$$ 
 is a projective resolution of $E_1 \boxtimes E_2$. Therefore, 
$\Ext^i_{G_1 \times G_2}[E_1\boxtimes E_2, F_1\boxtimes F_2]$ can be calculated by taking the cohomology of the 
chain complex  $\Hom_{G_1 \times G_2}[\bigoplus_{i+j = k} P_i\boxtimes Q_j, F_1\boxtimes F_2]$.

It is possible to choose a projective resolution of $E_1$ by  $P_i = {\rm ind}_{K_i}^{G_1}W_i$ for 
finite dimensional representations
$W_i$ of compact open subgroups $K_i$ of $G_1$. The existence of such a projective resolution 
 is made possible  through the construction of an 
equivariant sheaf on the Bruhat-Tits building of $G_1$ associated to the representation $E_1$,
cf. \cite{schneider-stuhler:
sheaves};
this is the step which needs $G_1$ to be reductive, and also requires  the admissibility of $E_1$.

Since $W_i$ are finite dimensional, we have the
isomorphism $$\Hom_{K \times G_2}[W_i \boxtimes Q_j, F_1\boxtimes F_2] \cong
\Hom_{K} [W_i, F_1] \otimes \Hom_{G_2}[Q_j, F_2],$$
therefore,
\begin{eqnarray*} 
\Hom_{G_1 \times G_2}[ P_i\boxtimes Q_j, F_1\boxtimes F_2] & = 
& \Hom_{G_1 \times G_2}[{\rm ind}_K^{G_1}(W_i) \boxtimes Q_j, F_1\boxtimes F_2] \\
&\cong & \Hom_{K \times G_2}[W_i \boxtimes Q_j, F_1\boxtimes F_2] \\
&\cong &  \Hom_{K} [W_i, F_1] \otimes \Hom_{G_2}[Q_j, F_2] \\
&\cong &  \Hom_{G_1} [P_i, F_1] \otimes \Hom_{G_2}[Q_j, F_2]. 
\end{eqnarray*}
Thus we are able to identify the chain complex  $\Hom_{G_1 \times G_2}[\bigoplus_{i+j = k} P_i\boxtimes Q_j, F_1\boxtimes F_2]$ as the tensor product of the chain complexes $\Hom_{G_1}[P_i,F_1]$ and $\Hom_{G_2}[Q_j,F_2]$.
Now the abstract Kunneth theorem which calculates the cohomology of the tensor product of two chain complexes
in terms of the cohomology of the individual chain complexes completes the proof of  the theorem. \end{proof}

\section{Branching laws from $\GL_{n+1}(F)$ to $\GL_n(F)$}
We begin by recalling
the following basic result in this context, cf. \cite{prasad2}. 

\begin{thm} \label{duke93}Given 
an irreducible generic representation $\pi_1$ of $\GL_{n+1}(F)$, and an
irreducible generic representation $\pi_2$ of $\GL_{n}(F)$, 
$$\Hom_{\GL_n(F)}[\pi_1,\pi_2] = \C.$$
\end{thm}

The aim of this section is to prove the following theorem which can be considered as the Euler-Poincar\'e version of Theorem \ref{duke93}. 

\begin{thm}\label{whittaker}
Let  $\pi_1$ be an admissible representation of  $\GL_{n+1}(F)$ of finite length, and 
$\pi_2$ an admissible representation of  $\GL_{n}(F)$ of finite length. 
Then, $\Ext^i_{\GL_n(F)}[\pi_1,\pi_2]$ are finite dimensional vector spaces over $\C$, and 
$$\EP_{\GL_n(F)}[\pi_1,\pi_2] 
= \dim {\rm Wh}(\pi_1) \cdot \dim {\rm Wh}(\pi_2),$$
where ${\rm Wh}(\pi_1)$, resp. ${\rm Wh}(\pi_2)$,  denotes the space of Whittaker models for $\pi_1$, resp. $\pi_2$,
with respect to fixed non-degenerate characters on the maximal unipotent subgroups
in $\GL_{n+1}(F)$ and $\GL_n(F)$.

\end{thm}

The proof of this theorem will be accomplished using some results of Bernstein and Zelevinsky regarding the
structure of representations of $\GL_{n+1}(F)$ restricted to the mirabolic subgroup.

Denote by $E_{n}$  the mirabolic subgroup of $\GL_{n+1}(F)$ 
consisting of matrices  whose last row is equal to $(0, 0,\cdots, 0, 1)$ and
let $N_{n+1}$ be the group of upper triangular unipotent matrices in 
$\GL_{n+1}( F)$. We will be using subgroups $\GL_i(F)$ of $\GL_{n+1}(F)$ for $i \leq n+1$ always sitting at the upper left corner of $\GL_{n+1}(F)$.
We fix a nontrivial character $\psi_0$ of $F$ and let 
$\psi_{n+1}$ be the character of $N_{n+1}$ given by
\[ \psi_{n+1}(u)=\psi_0(u_{1,2}+u_{2,3}+\cdots + u_{n,n+1}).\] 
For a representation $\pi$ of $\GL_{n+1}(F)$, let  
\[  \pi^i = \text{ the {\it i}-th derivative of $\pi$}, \]
which is a representation of $\GL_{n+1- i}( F)$. 
It will be important for us to note that $\pi^i$ are representations of finite length of $\GL_{n+1-i}(F)$ 
if $\pi$ is of finite length for $\GL_{n+1}(F)$.

To recall the definition of $\pi^i$,  let $  R_{n+1-i} = \GL_{n+1- i}(F) \cdot V_i $ be  the subgroup of
$\GL_{n+1}(F)$ consisting of matrices 
$$\left(\begin{array}{cc} g & v \\  0 & z \end{array} \right)$$ 
with $g\in \GL_{n+1-i}( F)$, $v\in M_{(n+1-i) \times i}$, $z\in N_i$.  If the
character $\psi_i$ of $N_i$ is extended to $V_i$ by extending it trivially across 
$M_{(n+1
-i) \times i}$,
then we have 
\[  \pi^i= \pi_{V_i, \psi_i} ,\]
where $\pi_{V_i,\psi_i}$ is the twisted Jacquet module of $\pi$, i.e., the maximal quotient of $\pi$ on which 
$V_i$ operates via the character $\psi_i$.

Here is a generality from Bernstein and Zelevinsky  \cite{bz1}, \S 3.5.

\begin{prop}  \label{filtration}
 Any  smooth representation $\Sigma$ of $E_{n}$ has a natural filtration by $E=E_{n}$
modules 
\[  0 =\Sigma_0 \subset  \Sigma_1 \subset  \Sigma_2 \subset \cdots \subset
\Sigma_{n+1} = \Sigma \]
such that
\[  \Sigma_{i+1}/\Sigma_i = {\rm ind}_{R_{i}}^{E_{n}}(\Sigma^{n+1-i} \otimes \psi_{n+1-i})  
\quad \text{for $i=0,\cdots, n$,} \] 
where $R_{i} = \GL_i(F)\cdot V_{n+1-i}$ is the subgroup of $\GL_{n+1}(F)$ consisting 
 of 
$$\left(\begin{array}{cc} g & v \\  0 & z \end{array} \right)$$ 
with $g\in \GL_i(F)$, $v\in M_{i \times(n+1-i)}$, $z\in N_{n+1-i}$, and the
character $\psi_{n+1-i}$ on $N_{n+1-i}$ is extended to $V_{n+1-i}$ by extending 
it trivially across 
$M_{i \times(n+1- i)}$.
\end{prop}

The proof of the following proposition is a direct consequence of Proposition \ref{bessel}.

\begin{prop} For a smooth representation $\pi_1$ of $\GL_{n+1}(F)$, and $\pi_2$ of $\GL_n(F)$,
$$\Ext^j_{\GL_n(F)}[{\rm ind}_{R_{i}}^{\GL_n(F)}(\pi_1^{n-i+1} \otimes \psi_{n-i }), {\pi}_2^\vee] =   \Ext^j_{\GL_{i}(F)}[{\pi}_2^{n-i}, ({\pi}_1^{n-i+1})^\vee].$$
\end{prop}

\begin{proof} Since $\GL_n(F)\cdot R_i = E_n$, $\GL_n(F) \cap R_i = \GL_i(F) \cdot V_{n-i}$, 
the restriction of $\pi_1^{n+1-i} \otimes \psi_{n+1-i}$ 
from $R_i$ to  $\GL_n(F) \cap R_i = \GL_i(F) \cdot V_{n-i}$ is $\pi_1^{n+1-i} \otimes \psi_{n-i}$ for
any $i$, $0 \leq i \leq n$. Therefore, the proposition follows from
Proposition \ref{bessel}.
\end{proof} 
 
\begin{lemma} \label{EPgln} For any two  smooth representations $V_1,V_2$ of $\GL_n(F)$, $n \geq 1$,  of finite length,  $$\EP_{\GL_n(F)}[V_1,V_2]=0.$$

\end{lemma}
\begin{proof} It suffices to prove the lemma assuming that both $V_1$ and $V_2$ are irreducible representations of $\GL_n(F)$. 
If the two representations $V_1,V_2$ were irreducible, and had  different central characters, then clearly
$\Ext^i_{\GL_n(F)}[V_1,V_2] = 0$ for all integers $i$. On the other hand, we know by Proposition \ref{prop:EP}(b) that for representations $V_1$ and $V_2$ of finite length 
of $\GL_n(F)$, $\EP_{\GL_n(F)}[V_1,V_2]$ 
is constant in a 
connected family, so denoting by $\nu$ the character $\nu(g)= |\det g|$ of $\GL_n(F)$, we have 
$\EP_{\GL_n(F)}[V_1,V_2] = \EP_{\GL_n(F)}[\nu^s\cdot V_1, V_2]$ for all $s \in \C$. Choosing $s$ appropriately, 
we can change the central character of $\nu^s \cdot V_2$ to be different from $V_1$, and hence  $\EP_{\GL_n(F)}[V_1,V_2]= 
\EP_{\GL_n(F)}[\nu^s \cdot V_1,V_2]=0$.
\end{proof}

\noindent{\bf Proof of Theorem \ref{whittaker}:} Since the Euler-Poincar\'e characteristic is additive in exact sequences, it suffices to calculate
$\EP_{\GL_n(F)}[{\rm ind}_{R_{i}}^{\GL_n(F)}(\pi_1^{n-i} \otimes \psi_{n-i +1}), \pi_2]$
which by Proposition \ref{bessel} is  $\EP_{\GL_{i}(F)}[({\pi}_2^{n-i})^\vee, ({\pi}_1^{n-i+1})^\vee],$ 
which by Lemma \ref{EPgln} above is $0$ unless $i=0$.  
(Note that 
in $\EP_{\GL_{i}(F)}[({\pi}_2^{n-i})^\vee, ({\pi}_1^{n-i+1})^\vee],$
both the representations involved are admissible representations
of $\GL_{i}(F)$.)
For $i=0$, note that we are dealing
with $\GL_0(F)= 1$, and the representations involved are  $(\pi_1 ^\vee)^{n+1}$ and $(\pi_2^\vee)^n$, which are nothing but the space of Whittaker models of $\pi_1^\vee$
and $\pi_2^\vee$. Since for representations $V_1, V_2$ of the group $\GL_0(F) = 1$, $\EP[V_1,V_2] = \dim \Hom[V_1,V_2] = \dim V_1 \cdot \dim V_2$, this completes the proof of the theorem.

\begin{remark} \label{nonvanishing}  One knows, cf. \cite{prasad2},
 that there are generic representations  of $\GL_{3}(F)$ which have the trivial
  representation of $\GL_2(F)$ as a quotient;  similarly, there are
  nongeneric representations of $\GL_{3}(F)$   with irreducible generic
  representations of $\GL_2(F)$ as a quotient. For such pairs $(\pi_1,\pi_2)$ of representations, it
  follows from Theorem \ref{whittaker} that   $\EP_{\GL_2(F)}[\pi_1, \pi_2]=0, $ whereas
  $\Hom_{\GL_2(F)}[\pi_1, \pi_2] \not = 0. $  Therefore, for such pairs $(\pi_1,\pi_2)$
  of representations,  we must have $\Ext^i_{\GL_2(F)}[\pi_1, \pi_2] \not =0, $ for some $i>0$.
  \end{remark}
\section{Conjectural vanishing of Ext groups for generic representations}

The following conjecture seems to be at the root of why the simple and general result of previous section on Euler-Poincar\'e
characteristic translates into   a simple result about Hom spaces for generic representations. The author has not managed to prove it
in any generality. There is a recent preprint by Chan and Savin, cf. \cite{savin}
dealing with some cases of this conjecture.

\begin{conj} \label{vanishing} Let $\pi_1$ be an irreducible generic representation of $\GL_{n+1}(F)$, and $\pi_2$ an
irreducible generic representation of $\GL_{n}(F)$. Then,
$$\Ext^i_{\GL_n(F)}[\pi_1,\pi_2] = 0, {\rm~~~for ~~~all~~~} i > 0.$$ \end{conj}

\begin{remark}
  By Remark \ref{nonvanishing}, one cannot remove the genericity condition for either $\pi_1$ or $\pi_2$ in the above conjecture.
  In particular, one cannot expect that a generic representation of $\GL_{n+1}(F)$
  when restricted to $\GL_n(F)$ is a projective
  representation in $\RR(\GL_n(F))$
  although this is the case for supercuspidal
  representations of $\GL_{n+1}(F)$.
  The paper \cite{savin} of Chan and Savin proves that the part of the
  Steinberg representation of $\GL_{n+1}(F)$ (denoted $\St_{n+1}$)
  in the Iwahori
  component of the Bernstein decomposition for $\RR(GL_n(F))$ is a projective module. There is no doubt then that $\St_{n+1}$ when restricted to $\GL_n(F)$
is a projective
representation in $\RR(\GL_n(F))$, therefore
$\Ext^i_{\GL_n(F)}[\St_{n+1},\pi_2] = 0$ for $i>0$ for any irreducible representation $\pi_2$ of $\GL_n(F)$. As a consequence, it will follow from the duality theorem of Schneider-Stuhler, cf. Theorem \ref{SS} below,
that $\St_{n+1}$ contains no irreducible submodule of $\GL_n(F)$.
  \end{remark}

Towards checking the validity of this conjecture in some cases, note that by Theorem \ref{duke93} and Theorem \ref{whittaker}, under the hypothesis of the conjecture,

$$  \dim \Hom_{\GL_n(F)}[\pi_1,\pi_2] = 1, {\rm~~~ and ~~~} \EP[\pi_1,\pi_2] = 1.
$$
It follows that if we already knew that 
$\Ext^i_{\GL_n(F)}[\pi_1,\pi_2] = 0, \,\,\, i>1,$
then we will also know that,
$\Ext^1_{\GL_n(F)}[\pi_1,\pi_2] = 0,$
and the conjecture will be proved for such representations.

It is easy to see that if $\pi_1$ or $\pi_2$ is cuspidal, then 
$\Ext^i_{\GL_n(F)}[\pi_1,\pi_2] = 0$ for $ i>1.$
We do one slightly less
obvious case when $\pi_1$ arises as a subquotient of a principal series
representation induced from a cuspidal representation of a maximal parabolic
in $\GL_{n+1}(F)$.

It follows from \cite{schneider-stuhler:sheaves}*{Corollary III.3.3(i)}
that
if $\pi_1$ arises from a cuspidal representation of a maximal parabolic
in $\GL_{n+1}(F)$, it has a projective resolution
of length 1 in the category $\RR_\chi(G)$
of smooth representations of $G=\GL_{n+1}(F)$ with central character $\chi$. It is easy to see
that a projective module in $\RR_{\chi}(\GL_{n+1}(F))$ when considered as a representation of
$\GL_n(F)$ is a projective module in $\RR(\GL_n(F))$, cf. Proposition 3.2 in \cite{NP}.
This proves vanishing of $\Ext^i_{\GL_n(F)}[\pi_1,\pi_2] = 0$ for $ i>1,$ hence also of
$\Ext^1_{\GL_n(F)}[\pi_1,\pi_2].$

This takes care of $G=\GL_{n+1}(F)$ for $n+1 \leq 3$, except that for $\GL_3(F)$
if both $\pi_1$ and $\pi_2$ arise as components of  principal series representations induced from their
Borel subgroups then there is a possibility of having nontrivial
$\Ext^2_{\GL_2(F)}[\pi_1,\pi_2]$. By the duality theorem of Schneider-Stuhler, cf. Theorem \ref{SS} below,
we will have, $\Hom_{\GL_2(F)}[D\pi_2,\pi_1] \not = 0$ (where $D\pi_2$ is the Aubert-Zelevinsky involution of $\pi_2$). The following proposition takes care of this.

\begin{prop}\label{gl3}
  Let $\pi_1$ be an irreducible generic representation of $\GL_3(F)$, and $\pi_2$ any irreducible representation of $\GL_2(F)$ which is not a twist of the
  Steinberg representation of $\GL_2(F)$. Then
  $$\Hom_{\GL_2(F)}[\pi_2,\pi_1]  = 0.$$
  \end{prop}

We will not prove  this proposition here but discuss two propositions which deal with all but a few cases of the proposition above. The cases left out by the next two
propositions can be handled by the Mackey restriction of an explicit principal series
(especially using that different inducing data can give rise to the same principal series).

\begin{prop} \label{nosub} Let $H_1 \subset H$ be $p$-adic groups with $Z=F^\times$ contained in the center of $H$ with $Z\cap H_1 = \{1\}$. Suppose $\mu$ is a smooth representation of $H_1$, and $\pi_2$ an irreducible admissible representation of $H$.
    Then $\Hom_{H}[\pi_2, \ind_{H_1}^{H}(\mu)]=0$.
  \end{prop}
\begin{proof}Note that for each $x \in H/H_1$,
  restriction of functions from
  $H$ to $(Z=F^\times)\cdot x$ gives rise to $F^\times$-equivariant  maps
  $$ \ind_{H_1}^{H}(\mu)
  \longrightarrow {\mathcal S}(F^\times),$$
    which can be assumed to be nonzero for any  $f \in \ind_{H_1}^{H}(\mu)$ by choosing
  $x \in H/H_1$ appropriately.

  Therefore if $\Hom_{H}[\pi_2, \ind_{H_1}^{H}(\mu)] \not =0$ choosing $x \in H/H_1$ appropriately,
  we get a nonzero map from $\pi_2 \longrightarrow {\mathcal S}(F^\times)$ which is
  $\omega$-equivariant where $\omega$ is the central character of $\pi_2$.
  Since ${\mathcal S}(F^\times)$ has no functions on which $F^\times$ operates by a character,
  the proof of the proposition is complete.
    \end{proof}

This lemma
when combined with the Bernstein-Zelevinsky filtration
in Proposition \ref{filtration} has the following as an
immediate consequence.

\begin{prop}\label{submodule}
  Let $\pi_1$ be any smooth representation of $\GL_{n+1}(F)$ of finite length,
  and $\pi_2$ any irreducible representation of $\GL_n(F)$. Then if
  $$\Hom_{\GL_n(F)}[\pi_2,\pi_1]  \not = 0,$$
  then $\pi_2$ appears a submodule of $J_{n,1}(\pi_1)$
  where $J_{n,1}$ denotes
  the (un-normalized) Jacquet module with respect to the $(n,1)$ parabolic in $\GL_{n+1}(F)$ considered
  as a module for $\GL_n(F) \subset \GL_n(F) \times \GL_1(F)$; in particular, if $J_{n,1}(\pi_1) =0$,
  then there are no nonzero $\GL_n(F)$-submodules in $\pi_1$.
  \end{prop}
\begin{proof} The proof of the proposition is an immediate consequence of the  observation
  that the Bernstein-Zelevinsky filtration 
  in Proposition \ref{filtration} when restricted to $\GL_n(F)$ gives rise to representations
  of $\GL_n(F)$ induced from subgroups $H_i \subset \GL_n(F)$ with $H_i \cap \{Z(\GL_n(F)) = F^\times \} = \{1\}$
  except in the case when $H_i=\GL_n(F)$ which corresponds to the Jacquet module   $J_{n,1}(\pi_1)$. \end{proof}

\section{Finite dimensionality of Ext groups}

In this section we prove the finite dimensionality of Ext-groups in the case of $\SO_n(F) \subset \SO_{n+1}(F)$. The proof will have an inductive
structure, and will involve Bessel models in the inductive step, so we begin by recalling the concept of Bessel models.

Let $V = X + D + W + Y$ be a quadratic space over the non-archimedean local field $F$ 
with $X$ and $Y$ totally isotropic subspaces of $V$ in duality with each other under the underlying bilinear form,  
$D$ an anisotropic line in $V$, 
and $W$ a quadratic subspace of $V$. Suppose
that the dimension of $X$ is $k$; fix  a complete flag 
$\langle e_1 \rangle \subset \langle e_1, e_2 \rangle \subset \cdots \subset \langle e_1,e_2,\cdots,e_k \rangle = X$ of isotropic subspaces in $X$. 
Let $P = MU$ be the parabolic subgroup in $\SO(V)$ stabilizing this flag,
with $M=\GL_1(F)^k \times \SO(D+W)$. 
For $W \subset V$ a codimension $2k+1$ subspace as above, the subgroup $\SO(W) \cdot U$ which is uniquely defined up to 
conjugacy by $\SO(V)$
makes frequent appearance in this work, as well as in other works on classical groups. 
We call this subgroup as the {\it Bessel subgroup}, and denote it as $\Bes(V,W) = \SO(W)\cdot U$.

Let $P_X= M_X \cdot U_X$ be the maximal parabolic of $\SO(V)$ stabilizing $X$. We have $M_X 
\cong \GL(X)\cdot \SO(W+D)$, and $U_X$ sits in the exact sequence,
$$1 \rightarrow \Lambda^2X \rightarrow U_X \rightarrow X \otimes ( D+W) \rightarrow 1.$$

Let $\ell: U \rightarrow F$ be a linear form such that
\begin{enumerate}
  \item its restriction to each of the simple root spaces in $\GL(X)$ 
defined by the flag $\langle e_1 \rangle \subset \langle e_1, e_2 \rangle \subset \cdots \subset \langle e_1,e_2,\cdots,e_k \rangle = X$ 
of isotropic subspaces in $X$
is non-trivial;
\item its restriction to the unipotent radical of the parabolic $P_X=M_XU_X$ in $\SO(V)$ stabilizing $X$ is trivial on the 
  subgroup of $U_X$ which is $\Lambda^2 X$;
 \item and on the quotient of $U_X$ by $\Lambda^2 X$ which can be identified to $(D+W) \otimes X$, $\ell$ is given by 
the tensor product of a 
linear form on $D+W$ which is trivial on $W$, and a linear form on $X$ which is trivial on the subspace $\langle e_1,e_2,\cdots,e_{k-1} \rangle$.
\end{enumerate}

Composing the linear form $\ell:U \rightarrow F$ with a nontrivial character $\psi_0:F \rightarrow \C^\times$, we get a character $\psi:U \rightarrow \C^\times$.
This character $\psi:U \rightarrow \C^\times$ 
depends only on $W \subset V$ a nondegenerate subspace of $V$ of odd codimension,
such that the quadratic space $V/W$ is split, and is independent of all choices made along the way (including that of the character $\psi_0$).
The character $\psi$ of $U$ is invariant under $\SO(W)$. For any representation $\sigma$ of $\SO(W)$, $\Bes(V,W) = \SO(W)\cdot U$
comes equipped with the representation which is
$\sigma$ on $\SO(W)$, and $\psi$ on $U$; since $\psi$ is fixed when considering representations of $\Bes(V,W) = \SO(W)\cdot U$, we denote this representation of
$\Bes(V,W) = \SO(W)\cdot U$ as $\sigma$ itself or sometimes as $\sigma \otimes \psi$.

The Bessel models of a smooth  representation $\pi$ of $\SO(V)$ are irreducible admissible representations  $\sigma$ of $\SO(W)$ such that,
$$\Hom_{\Bes(V,W)}[\pi,\sigma ]  \cong  
\Hom_{\SO(V)} \left [\pi, {\rm Ind}_{\Bes(V,W)}^{\SO(V)} (\sigma ) \right ]
\cong 
\Hom_{\SO(V)} \left [ {\rm ind}_{\Bes(V,W)}^{\SO(V)} (\sigma^\vee ), \pi^\vee \right ] \not =  0.
$$
When $W$ is a codimension one subspace of $V$, then
$\Bes(V,W)=W$, and 
the notion of a Bessel model is simply that of restriction from $\SO(V)$ to $\SO(W)$, whereas 
when $\dim(W) = 0,1$, then the notion of a Bessel model is nothing but that of the Whittaker model (for a particular character of the maximal unipotent subgroup of $\SO(V)$ if
$\dim(W) =1$). 

We can define the higher Ext versions of the Bessel models as,
$$\Ext^i_{\Bes(V,W)}[\pi,\sigma]  
\cong  \Ext^i_{\SO(V)} \left [\pi, {\rm Ind}_{\Bes(V,W)}^{\SO(V)} (\sigma ) \right ]  
 \cong  
\Ext^i_{\SO(V)} \left [ {\rm ind}_{\Bes(V,W)}^{\SO(V)} (\sigma^\vee), \pi^\vee \right ].$$

The following proposition whose proof we will omit is analogous to that of
Theorem 15.1 of \cite {ggp}. It 
allows one to prove finite dimensionality of 
$\Ext^i_{\Bes(V,W_0)}[\pi,\sigma ]$ if we know the finite dimensionality of 
$\Ext^i_{\SO(W)}[\pi,\sigma'] $ where $W$ is a codimension one subspace in the quadratic space $V$, and $\sigma'$ an irreducible representation of $\SO(W)$.

\begin{prop}
Let $W \subset V$ be a nondegenerate quadratic subspace of codimension 1
over a non-archimedean local field $F$. Suppose that
\[  W=Y_k\oplus  W_0 \oplus Y_k^\vee\]
and 
\[ V= Y_k \oplus V_0 \oplus Y_k^\vee, \]
with $Y_k$ and $Y_k^\vee$ isotropic subspaces and $W_0 \subset V_0$ nondegenerate quadratic spaces with $W_0$ a subspace of codimension one in
$V_0$. Let
$P_W(Y_k)$ be the parabolic in $\SO(W)$ stabilizing $Y_k$ with Levi subgroup
\[  M = \GL(Y_k)  \times \SO(W_0) \]
For an irreducible supercuspidal
representation $\tau$ of $\GL(Y_k)$ and an irreducible admissible
representation $\pi_0$ of $\SO(W_0)$, let 
\[  \tau \rtimes \pi_0 = \Ind_{P_W(Y_k)}^{\SO(W)} (\tau \boxtimes \pi_0) \] 
be the corresponding (un-normalized) principal series representation of $\SO(W)$. Let
$\pi$ be an irreducible admissible representation $\SO(V)$ which does
not belong to the Bernstein component associated to $(\GL(Y_k) \times \SO(V_0), \tau \boxtimes \mu)$
for any irreducible representation $\mu$ of $\SO(V_0)$. Then
\[ \Ext^i_{\SO(W)}[\pi, \tau \rtimes \pi_0] \cong 
\Ext^i_{\Bes(V,W_0)}[\pi, \pi_0  ]. \]
\end{prop}

\begin{cor} With the notation as above, if 
$ \Ext^i_{\SO(W)}[\pi, \tau \rtimes  \pi_0]$ are finite dimensional, 
then so are $\Ext^i_{\Bes(V,W_0)}[\pi, \pi_0  ]$.
\end{cor}
\begin{proof}It suffices to observe that given $\pi_0$, there is a representation $\pi$ of $\SO(V)$ which does not belong to the Bernstein component associated to $(\GL(Y_k) \times \SO(V_0), \tau \boxtimes 
\mu)$ for any irreducible representation $\mu$ of $\SO(V_0)$. 
\end{proof}

We now come to the proof of finite dimensionality of the $\Ext$ groups.

\begin{thm} \label{finite}
  Let $V = X + D + W + Y$ be as at the beginning of the section,
   a quadratic space over the non-archimedean local field $F$ with $W$ a quadratic subspace of codimension $2k+1$. 
Then for any  irreducible admissible representation $\pi$ of $\SO(V)$ 
and irreducible admissible representation $\sigma$ of $\SO(W)$,  $\Ext^i_{\Bes(V,W)}[\pi,\sigma ]$ are finite dimensional vector
spaces over $\C$ for all $i \geq 0$.
\end{thm}

\begin{proof} The proof of this theorem will be by induction on the dimension of $V$. We thus assume that 
  for any quadratic spaces ${\mathcal W} \subset {\mathcal V}$ with $\dim({\mathcal V}) < \dim (V)$ (with ${\mathcal V}/{\mathcal W}$ a split quadratic space of
  odd dimension), and
for any  irreducible admissible representation $\pi$ of $\SO({\mathcal V})$ 
and irreducible admissible representation $\sigma$ of $\SO({\mathcal W})$,
$$\Ext^i_{\Bes({\mathcal V}, {\mathcal W})}[\pi,\sigma ]$$ are finite dimensional vector
spaces over $\C$ for all $i \geq 0$.

We begin by  proving the theorem for a principal series representation of $\SO(V)$ induced from an irreducible 
representation of a maximal  parabolic subgroup. By the previous proposition, we need only prove the finite dimensionality of
$\Ext^i_{\SO(V')}[\pi,\pi']$ 
where $V'$ is a codimension one subspace of $V$, and $\pi'$ is an irreducible, admissible
representation of $\SO(V')$.

Much of the proof below  closely follows the paper of Moeglin-Waldspurger \cite{wald:gp3}, 
where they have to do much harder work to precisely analyze $\Hom_{\SO(V')}[\pi,\pi']$.

Assume  that the  dimension of $V$ is $n+1$, and that $V'$ is a subspace of dimension $n$. Let $V = X + V_0 + Y$ 
with $X$ and $Y$ totally isotropic subspaces of $V$ of dimension $m$, and in perfect pairing with each other.
Let $P$ be the maximal parabolic subgroup of $\SO(V)$ stabilizing $X$. Let $M = \GL(X) \times \SO(V_0)$ be a Levi subgroup of $P$, 
$ \pi_0 \otimes \sigma_0$ an irreducible representation of $M$ 
realized on the space $E_{\pi_0} \otimes E_{\sigma_0}$, and $\pi = \pi_0 \rtimes \sigma_0$ the corresponding principal series representation of $\SO(V)$. Denote by  $E_\pi$ the space of  function on $\SO(V)$ with values in
$E_{\pi_0} \otimes E_{\sigma_0}$
verifying the usual conditions under left translation by $P(F)$ 
for defining the principal series representation $\pi = \pi_0 \rtimes \sigma_0$  of $\SO(V)$.

To understand the restriction of the principal series $\pi = \pi_0 \rtimes \sigma_0$ to $\SO(V')$, we need to analyze the orbits of $\SO(V')$ on $P(F)\backslash \SO(V)$.
To every $g \in P(F)\backslash \SO(V)$, one can associate an isotropic subspace $g^{-1}(X)$ of $V$. Let ${\mathcal U}$ be the set of $g \in P(F)\backslash \SO(V)$ such that
$\dim(g^{-1}(X) \cap V') = m-1$, and let ${\mathcal X}$ be the set of $g \in P(F)\backslash \SO(V)$ such that
$\dim(g^{-1}(X) \cap V') = m$.  Then ${\mathcal U}$ is an open subset of $P(F)\backslash \SO(V)$ which is a single orbit under 
$\SO(V')$, and  ${\mathcal X}$ is a closed subset of $P(F)\backslash \SO(V)$ which is a single orbit under $\SO(V')$ 
unless $n$ is even, and $n =2m$ in which case there are two orbits in ${\mathcal X}$ under $\SO(V')$.

Denote
by  $E_{\pi, {\mathcal U}}$ the subspace of  functions in $E_\pi$ with support in ${\mathcal U}$, and denote by $E_{\pi, {\mathcal X}}$ 
the space $E_\pi/E_{\pi, {\mathcal U}}$. The spaces $E_{\pi, {\mathcal U}}$ and $E_{\pi, {\mathcal X}}$ are invariant under $\SO(V')$, and we have an exact sequence of $\SO(V')$-modules,
$$0 \rightarrow E_{\pi, {\mathcal U}} \rightarrow E_\pi \rightarrow E_{\pi, {\mathcal X}} \rightarrow 0.$$ 

To prove the finite dimensionality of Ext groups $\Ext^i_{\SO(V')}[E_\pi, \pi']$,
it suffices to prove similar finite dimensionality theorems for the Ext groups involving  the $\SO(V')$-modules $ E_{\pi, {\mathcal U}}$ and  $E_{\pi, {\mathcal X}}$. We analyze the two terms separately.

For analyzing $E_{\pi, {\mathcal X}}$, 
we assume (after conjugation by $\SO(V)$) that both $X$ and $Y$ are contained in $V'$. Thus, $V' = X + V'_0 + Y$ with $V_0' = V'\cap V_0$.
It can be seen that, 
$$E_{\pi, {\mathcal X}} = \pi_0|.|^{1/2} \times \sigma_0|_{\SO(V'_0)}.$$

By Theorem \ref{second-ad} (the second adjointness theorem of Bernstein),
$$\Ext^i_{\SO(V')}[E_{\pi, {\mathcal X}}, \pi'] = \Ext^i_M[\pi_0|.|^{1/2} \boxtimes \sigma_0|_{\SO(V'_0)}, \pi'_{N^-}],$$
where $M=\GL(X) \times \SO(V_0')$.

 The proof of the finite dimensionality of $\Ext^i_{\SO(V')}[E_{\pi, {\mathcal X}}, \pi' ]$ 
 now follows from the induction hypothesis according to which the theorem was supposed to be known
for $\SO(V'_0) \subset \SO(V_0)$, besides the fact that $\pi'_{N^-}$, the Jacquet module with respect to the opposite 
parabolic $P^- = M N^-$ is an admissible representation of $M = \GL(X) \times \SO(V_0')$, 
hence has a finite filtration by 
tensor product of irreducible representations of $\GL(X)$ and $\SO(V_0')$, and then an application of the Kunneth theorem.

We now move on to $E_{\sigma, {\mathcal U}}$. In this case, after conjugation by $\SO(V)$, we will be in the situation,
$$V' = X' + D' + V_0 + Y',$$
where $X'$ and $Y'$ are totally isotropic subspaces of $V'$ of dimension $(m-1)$, and
$D', V_0, X'+Y'$ are non-degenerate quadratic 
spaces. Let $X'=\{ e_1,\cdots,e_{m-1}\}$, and $Y'=\{f_1,\cdots, f_{m-1} \}$. Let
$V'_k=X'_{k-1} + D' + V_0 + Y'_{k-1},$
where $X'_{k-1} = \{e_{m-k+1},\cdots,e_{m-1}\}$, and $Y'_{k-1} = \{f_{m-k+1},\cdots, f_{m-1}\}$, each of dimension $(k-1)$ for $k=1,\cdots, m$, and let $G'_k = \SO(V'_k).$

Using the filtration of the representation $\pi_0$ of $\GL(X)$ restricted to its mirabolic subgroup
given by Proposition \ref{filtration}  
in terms of the derivatives $\pi_0^k=\Delta^k(\pi_0)$,  Moeglin-Waldspurger \cite{wald:gp3} obtain a filtration, 
$$0 = \mu_{m+1} \subset \mu_m \subset \mu_{m-1} \cdots \subset \mu_1 =E_{\sigma, {\mathcal U}},$$  
with,
$$\mu_k/\mu_{k+1} \cong \Delta^k(\pi_0) \rtimes \mu'_k,$$
as modules for $\SO(V')$; the representation $\Delta^k(\pi_0) \rtimes \mu'_k$ is a principal series representation of $\SO(V')$ induced from a parabolic $P_k=M_kN_k$ 
with Levi subgroup $M_k= GL_{m-k}(F) \times G'_k$, and  where 
$$\mu_k' = {\rm ind}^{G'_k}_{\Bes(V'_k,V_0)}( \sigma_0 ),$$

By Theorem \ref{second-ad},
$$\Ext^i_{\SO(V')}[\mu_k/\mu_{k+1}, \pi'] = \Ext^i_{\SO(V')}[\Delta^k(\pi_0) \rtimes \mu'_k, \pi']
\cong \Ext^i_{M_k}[\Delta^k(\pi_0) \boxtimes \mu'_k , \pi'_{{N_k}^-}],$$ 
with $M_k= GL_{m-k}(F) \times G'_k$, a Levi subgroup in $\SO(V')$. Once again Kunneth theorem 
implies the finite dimensionality
of the Ext groups,
$$\Ext^i_{M_k}[\Delta^k(\pi_0) \boxtimes \mu'_k , \pi'_{{N^-_k}}].$$ 

Having proved the theorem for principal series representations of $\SO(V)$  induced  from maximal parabolics, 
the next lemma  proves the theorem in general.\end{proof}

\begin{lemma}\label{6.1}
Let $V $ be a quadratic space over the non-archimedean local field $F$ with $W$ a quadratic subspace of codimension 1. 
If for any principal series  representation $Ps$ of $\SO(V)$ induced from a maximal parabolic
and any irreducible admissible representation $\sigma$ of $\SO(W)$,  $\Ext^i_{\SO(W)}[Ps,\sigma]$ are 
finite dimensional vector
spaces over $\C$ for all $i \geq 0$, then for any irreducible   representation $\pi$ of $\SO(V)$ 
and any irreducible admissible representation $\sigma$ of $\SO(W)$,  $\Ext^i_{\SO(W)}[\pi,\sigma ]$ are 
finite dimensional vector
spaces over $\C$ for all $i \geq 0$.

\end{lemma}
\begin{proof} 
By Proposition \ref{closed}, if $\pi$ is a supercuspidal representation of $\SO(V)$, its restriction to $\SO(W) $  is a projective object in $\RR(SO(W))$. Therefore $\Ext^i_{SO(W)}[\pi,\sigma]$ are zero for $i>0$, and $\Ext^0_{SO(W)}[\pi,\sigma] = \Hom_{SO(W)}[\pi,\sigma]$ is finite dimensional. 

Assume now that $\pi$ is not a supercuspidal representation, and
write $\pi$ as a quotient of a principal series representation induced from a representation of a maximal parabolic subgroup of $\SO(V)$. We thus have an exact sequence,
$$0 \rightarrow \lambda \rightarrow Ps \rightarrow \pi  \rightarrow 0.$$
This gives rise to a long exact sequence,
\begin{eqnarray*}
0 & \rightarrow&  \Hom_{\SO(W)}[\pi, \sigma] 
\rightarrow  \Hom_{\SO(W)}[Ps, \sigma]  \rightarrow \Hom_{\SO(W)}[\lambda, \sigma]  \rightarrow \\  
& \rightarrow &  \Ext^1_{\SO(W)}[\pi, \sigma]   \rightarrow \Ext^1_{\SO(W)}[Ps, \sigma]  
\rightarrow  \Ext^1_{\SO(W)}[\lambda, \sigma]  \rightarrow \Ext^2_{\SO(W)}[\pi, \sigma]  \rightarrow \cdots  
\end{eqnarray*}

Since we know that all Hom spaces in the above exact sequence are finite dimensional, and also
$\Ext^1_{\SO(W)}[Ps, \sigma]  $ is given to be finite dimensional, 
we get the 
finite dimensionality of $\Ext^1_{\SO(W)}[\pi, \sigma]  $ for any irreducible representation $\pi$ of $\SO(V)$. 
This implies finite dimensionality of 
$\Ext^1_{\SO(W)}[\pi, \sigma]  $ for any finite length  representation $\pi$ of $\SO(V)$. 
Armed with this finite dimensionality of $\Ext^1_{\SO(W)}[\pi, \sigma]  $ for any finite length  
representation $\pi$ of $\SO(V)$, and with the knowledge that $\Ext^2_{\SO(W)}[Ps, \sigma]  $ is given to be 
finite dimensional, we get the finite dimensionality of $\Ext^2_{\SO(W)}[\pi, \sigma]  $ for any irreducible 
representation $\pi$ of $\SO(V)$, and similarly we get the finite dimensionality of 
$\Ext^i_{\SO(W)}[\pi, \sigma]  $ for any irreducible representation $\pi$ of $\SO(V)$, and any $i \geq 0$.
\end{proof}

The proof of Theorem \ref{finite} uses Theorem \ref{kunneth} (Kunneth theorem) for
representations of $\GL_m(F) \times G'_k$. Since
for any two irreducible representations $V,V'$ of $\GL_m(F)$, $\EP_{\GL_m(F)}[V,V']=0$ unless $m=0$ cf. Lemma \ref{EPgln}, we obtain the following corollary of the proof of the theorem.

\begin{cor}
  For a principal series representation
  $\pi = \pi_0\rtimes \sigma_0$ of $\SO(V)$ where $\sigma_0$ is an admissible representation of $\SO(W)$, and $\pi'$ is an
  admissible representation of $\SO(V')$ where $W \subset V'\subset V$
  with $V'$ a nondegenerate codimension 1 
  subspace of the quadratic space $V$ with $\dim(V')-\dim(W)=2m-1$,
  $$\EP_{\SO(V')}[\pi,\pi'] = \EP_{\Bes(V', W)}[\pi',\sigma_0 ] \cdot \dim \Wh(\pi_0).$$ 
  \end{cor}

\noindent{\bf Definition:} A finite length representation $\pi$ of a classical group
will be called a {\it full principal series} if it is irreducible and
supercuspidal, or is
of the form $\pi = \pi_0 \rtimes \sigma_0$ with both $\pi_0$ and $\sigma_0$ irreducible, and $\sigma_0$ supercuspidal.

\vspace{2mm}

The following corollary is a consequence of the previous corollary together with the
fact that if $\sigma$ is a cuspidal representation of $\SO(W)$, then $\sigma \otimes \psi$ is an injective module for $\Bes(V,W)$.

\begin{cor}\label{epson}
 Let $\pi$ be a finite length
  representation  of $\SO(V)$,
    and $\pi'$ of $\SO(V')$ where $ V'\subset V$ is a  nondegenerate codimension 1 
  subspace of the quadratic space $V$. Assume that
  $\pi$  is a full principal series, and $\pi'$
  is an irreducible representation of $\SO(V')$. Then,
  $\EP_{\SO(V')}[\pi,\pi']$
  is either 0 or 1.
    If $\pi = \pi_0\rtimes \sigma_0$ of $\SO(V)$
    where $\sigma_0$ is an admissible representation of $\SO(W)$ with $\dim W \leq 1$,   $\EP_{\SO(V')}[\pi,\pi'] = \dim \Wh(\pi) \cdot \dim \Wh(\pi')$ (if $\dim W=1$, $\Wh(\pi')$ is for a particular character of a maximal unipotent subgroup of $\SO(V')$).
\end{cor}

\begin{remark}
In the previous corollary, we see a large number of cases when the Euler-Poincar\'e characteristic is 0 or 1. Is there a  multiplicity one result for $\EP$, or for $\Ext^i$?
  \end{remark}

\section{An integral formula of Waldspurger, and a  conjecture on E-P}

In this section we review an integral formula of Waldspurger which we then propose to be the 
integral formula for the Euler-Poincar\'e pairing for $\EP_{\Bes(V,W)}[\sigma,\sigma' ]$ for $\sigma$ 
any finite length
representation of $ \SO(V)$, and $\sigma'$ any finite length representation of $\SO(W)$,
where $V$ and $W$ are quadratic spaces over $F$ with 
$V = X + D + W + Y$  with $W$ a quadratic subspace of $V$ of codimension $2k+1$ 
with $X$ and $Y$ totally isotropic subspaces of $V$ in duality with each other under the underlying bilinear form,  
and $D$ an anisotropic line in $V$. Let $Z=X+Y$.

Let $\underline{\mathcal T}$ denote the set of elliptic tori $T$ in $\SO(W)$ such that there exist quadratic subspaces  $W_T,W'_T$ of $W$ such that:
\begin{enumerate}

\item $W= W_T \oplus W'_T$, and $V=W_T \oplus W'_T \oplus D \oplus Z$.

\item $\dim (W_T)$ is even, and $\SO(W'_T)$ and 
$\SO(W'_T \oplus D \oplus Z)$ are  quasi-split.

\item $T$ is a maximal (elliptic) torus in $\SO(W_T)$.

\end{enumerate}

Clearly the group $\SO(W)$ operates on $\underline{\mathcal T}$. Let
${\mathcal T}$ denote a set of orbits for this action of $\SO(W)$  on $\underline{\mathcal T}$. 
For our purposes we note the
most important elliptic torus $T= \langle e \rangle$ 
corresponding to $W_T= 0$.

For $\sigma$ an admissible representation of $\SO(V)$ of finite length, define a function 
$c_\sigma(t)$ for regular elements of a  torus $T$ belonging to $\underline{\mathcal T}$ by the germ
expansion of the character $\theta_\sigma(t)$ of $\sigma$ on the centralizer of $t$ in the Lie algebra 
of $\SO(V)$, and picking out `the' leading term. The semi-simple part of the centralizer of $t$ in the 
Lie algebra of $V$ is the Lie algebra of $\SO(W'_T\oplus D \oplus Z)$ which,  if $W'_T \oplus D \oplus Z$ 
has odd dimension, has a 
unique conjugacy class of regular nilpotent element, but if $W'_T \oplus D \oplus Z$ has even dimension, then although
there are several regular nilpotent conjugacy classes, there is one which is `relevant', and is what is used
to define $c_\sigma(t)$. Similarly, for $\sigma'$ an admissible representation of $\SO(W)$ of finite length, 
one defines a function 
$c_{\sigma'}(t)$ for regular elements of a  torus $T$ belonging to $\underline{\mathcal T}$ by the germ
expansion of the character $\theta_{\sigma'}(t)$ of $\sigma'$.

Define a function $\Delta_T$ on an elliptic torus $T$ belonging to
$\underline{\mathcal T}$ with $W= W_T \oplus W'_T$, by $\Delta(t) = |\det(1-t)|_{W_T}|_F,$ and let  $D^H$ denote the discriminant function 
on $H(F)$. For a torus $T$ in $H$, define the Weyl group $W(H,T)$ by the usual normalizer divided by the centralizer:
$W(H,T) = N_{H(F)}(T)/Z_{H(F)}(T)$.

The following theorem is proved by Waldspurger
in  \cite{wald:gp} and \cite{wald:gp1}. 
\begin{thm} 
Let $V = X + D + W + Y$ be a quadratic space over the non-archimedean local field $F$ with $W$ a quadratic subspace of codimension $2k+1$ as above. 
Then for any  irreducible admissible representation $\sigma$ of $\SO(V)$ 
and irreducible admissible representation $\sigma'$ of $\SO(W)$,  
$$\sum _{T \in {\mathcal T}} |W(H,T)|^{-1} \int_{T(F)} c_\sigma(t) c_{\sigma'}(t) D^H(t) \Delta^k(t) dt ,$$
is a finite sum of absolutely convergent integrals. (The Haar measure on $T(F)$ is normalized to have volume 1.) If either $\sigma$ is a supercuspidal representation of 
$\SO(V)$, 
and $\sigma'$ is arbitrary irreducible admissible representation of $\SO(W)$,
or both $\sigma$ and $\sigma'$ are tempered representations, then 
$$\dim \Hom_{\Bes(V,W)}[\sigma,\sigma' ] =  
\sum _{T \in {\mathcal T}} |W(H,T)|^{-1} \int_{T(F)} c_\sigma(t) c_{\sigma'}(t) D^H(t) \Delta^k(t)dt.$$
\end{thm}

Given this theorem of Waldspurger, it is most natural to propose the following conjecture on
Euler-Poincar\'e pairing.

\begin{conj} \label{integral}
Let $V = X + D + W + Y$ be a quadratic space over the non-archimedean local field $F$ with $W$ a quadratic subspace of 
$V$ of codimension $2k+1$ as before.  Then for any  irreducible admissible representation $\sigma$ of $\SO(V)$ 
and irreducible admissible representation $\sigma'$ of $\SO(W)$,  

\begin{enumerate}
\item 
\begin{eqnarray*} 
\EP_{\Bes(V,W)}[\sigma, \sigma' ] &= & 
 \sum_i (-1)^i \dim \Ext^i_{\Bes(V,W)}[\sigma,\sigma'] \\ 
&= &  
\sum _{T \in {\mathcal T}} |W(H,T)|^{-1} \int_{T(F)} c_\sigma(t) c_{\sigma'}(t) D^H(t) \Delta^k(t) dt.
\end{eqnarray*}

\item If $\sigma$ and $\sigma'$ are  irreducible tempered representations, then
$\Ext^i_{\Bes(V,W)}[\sigma,\sigma'] = 0$ for $i > 0$.
  
\end{enumerate}

\end{conj}

\begin{remark}Note that a supercuspidal representation of $\SO(V)$ is a projective object in the 
category of smooth representations of $\SO(V)$, and hence by Proposition 2.4, it remains a projective
object in the category of smooth representations of $\SO(W) \cdot U$. Therefore if $\sigma$ or $\sigma'$ is 
supercuspidal,  $\Ext^i_{\Bes(V,W)}[\sigma,\sigma'] = 0$ for $i > 0$. (We note that for a supercuspidal representation $\sigma'$ of $\SO(W)$, 
the representation $\sigma' \otimes \psi$ is an injective module in the category of smooth representations
of $\SO(W) \cdot U$.) Thus Waldspurger's theorem 
is equivalent to the
conjectural statement on Euler-Poincar\'e characteristic if $\sigma$ or $\sigma'$ is supercuspidal (except that it is not proved if $\sigma'$ is supercuspidal, but $\sigma$ is arbitrary). \
Part 2 of the conjecture is there as the simplest 
possible explanation of Waldspurger's theorem for tempered representations!
\end{remark}

\begin{example}
Assume that either $G=\SO_{n+1}(F)$  is a split group, and $\sigma$ is induced from a character of a Borel subgroup of $G$,
or $H= \SO_{n}(F)$ is a split group and  $\sigma'$  is induced from a character of a Borel subgroup of $H$. 
Then  the conjectural formula on Euler-Poincar\'e becomes
$\EP[\sigma, \sigma']= 1$ which is a consequence of Corollary \ref{epson}.
\end{example}

\begin{remark} We consider the {Waldspurger integral formula} as some kind of
  {Riemann-Roch theorem.} Recall that for
  $X$ a smooth projective variety with Todd class $T_X$, and for any coherent sheaf $\F$ on $X$ with
Chern class $c(\F)$, one has,
$$EP(X,\F) = \sum_i(-1)^i \dim H^i(X,\F) = \sum_i(-1)^i \dim \Ext^i(\OO_X,\F) = \int_X(T_X \cdot c(\F)).$$
In our case,  $ EP[\pi_1,\pi_2] = \sum_i (-1)^i \dim \Ext^i_H[\pi_1,\pi_2]$ 
is conjecturally expressed as 
$$ EP[\pi_1,\pi_2] = \int_X T_X \cdot c(\pi_1,\pi_2),$$
where $X$ is a certain set of elliptic tori in $H$, $T_X$ is a function on this set of elliptic tori, and $c(\pi_1,\pi_2)$ is a function
on these elliptic tori defined in terms of the germ expansion of $\pi_1$ and $\pi_2$.
\end{remark}

\section{The Schneider-Stuhler duality theorem} \label{section10}
The following theorem is a mild generalization of a
duality theorem of Schneider-Stuhler due to M. Nori and this author in \cite{NP}; it turns
questions on $\Ext^i[\pi_1,\pi_2] $ to $\Ext^j[\pi_2,\pi_1]$, and therefore is of central importance to our theme in this paper.

\begin{thm} \label{SS}
  Let $G$ be a reductive $p$-adic group, and $\pi$ an irreducible, admissible representation of $G$. 
Let $d(\pi) $ be the   largest integer $i \geq 0$  such that there is an irreducible, admissible representation 
 $ \pi'$  of $G$  with ​
$\Ext^i_G[\pi,\pi']$  
nonzero​.  Then,

\begin{enumerate}
\item There is a unique irreducible representation $\pi'$ of $G$ with $\Ext^{d(\pi)}_G[\pi,\pi'] \not = 0.$   
\item The representation $\pi'$ in $(1)$ 
is nothing but $D(\pi)$  where $D(\pi)$ is the Aubert-Zelevinsky involution of $\pi$, and $d(\pi)$ is the split rank of the Levi subgroup $M$ of $G$ which carries the cuspidal support of $\pi$. 
\item $\Ext_G^{d(\pi)}[\pi, D(\pi)] \cong \C$. 
\item For any smooth  representation $\pi'$ of $G$, the bilinear pairing
$$(*) \,\,\,\,\,   \Ext^{i}_G[\pi, \pi']  \times \Ext^{j}_G[\pi', D(\pi)]  \rightarrow 
\Ext^{i+j = d(\pi)}_G[\pi, D(\pi)]   \cong \C, $$
is nondegenerate in the sense that if $\pi' = \displaystyle{\lim_{\rightarrow}\, } \pi'_n$ of
finitely generated $G$-sub-modules $\pi'_n$,
then $\Ext^{i}_G[\pi, \pi'] =\displaystyle{\lim_{\rightarrow}\, }
\Ext^{i}_G[\pi, \pi_n']$,
a direct limit of finite dimensional vector spaces over $\C$, and 
$\Ext^{j}_G[\pi', D(\pi)]  = \displaystyle{\lim_{\leftarrow}\, } \Ext^{j}_G[\pi'_n, D(\pi)]  $, an inverse limit
of finite dimensional vector spaces over $\C$, and the pairing  in $(*)$ is the direct limit of
perfect pairings on these finite dimensional spaces:
$$\Ext^{i}_G[\pi, \pi_n']  \times \Ext^{j}_G[\pi'_n, D(\pi)]  \rightarrow 
\Ext^{i+j = d(\pi)}_G[\pi, D(\pi)]   \cong \C.$$
(Observe that a compatible family of perfect pairings on finite dimensional vector spaces
$B_n: V_n \times W_n \rightarrow \C$ with $V_n$ part of an inductive system, and $W_n$ part of a projective system, gives rise to a natural pairing $B :
\displaystyle{\lim_{\rightarrow}\, }V_n \times \displaystyle{\lim_{\leftarrow}\, } W_n \rightarrow \C$
such that the associated homomorphism
from $(\displaystyle{\lim_{\rightarrow}\, }V_n)^\star$ to  $\displaystyle{\lim_{\leftarrow}\, } W_n$ is an isomorphism.
\end{enumerate}
\end{thm}

As an example, the 
following proposition giving complete classification
of irreducible submodules $\pi$ of the tensor product $\pi_1 \otimes \pi_2$ of two (irreducible, infinite dimensional)
representations $\pi_1,\pi_2$ of $\GL_2(F)$ with the product of their central characters trivial, is essentially a  translation of vanishing of
$\Ext^1_{PGL_2(F)}[\pi_1 \otimes \pi_2,\pi_3]$
by this duality theorem. The vanishing itself follows because 
$\Ext^2_{PGL_2(F)}[\pi_1 \otimes \pi_2,\pi_3]=0$ by Proposition \ref{ext-ps}, 
and both $\EP$ and Hom spaces have the same dimension.

\begin{prop}
Let $\pi_1, \pi_2$ be two irreducible admissible infinite dimensional 
representations of $\GL_2(F)$ with product of their
central characters trivial. Then the following is the complete list of irreducible sub-representations
$\pi$ of $\pi_1 \otimes \pi_2$ as $\PGL_2(F)$-modules.

\begin{enumerate}
\item $\pi$ is a supercuspidal representation of $\PGL_2(F)$, and appears as a quotient of $\pi_1 \otimes \pi_2$.

\item $\pi$ is a twist of the Steinberg representation, which we assume by absorbing the twist in 
$\pi_1$ or $\pi_2$ to be the Steinberg representation $\St$ of $\PGL_2(F)$. Then $\St$ is a 
submodule of $\pi_1 \otimes \pi_2$ if and only if 
$\pi_1,  \pi_2$ are both irreducible
principal series representations, and $\pi_1 \cong \pi_2^\vee$.
\end{enumerate}
\end{prop}

\section{Geometrization of Ext groups} \label{geometrization}

A natural way to construct exact sequences
in representation theory  is via the Bernstein-Zelevinsky exact sequence arising from the 
inclusion of an open set $X-Y$ in a topological space $X$ equipped with an $\ell$-sheaf $\frak F$, with $Y$ a closed subspace of $X$, giving rise to 
$$0\rightarrow \Sc(X-Y, {\frak F})\rightarrow \Sc(X, {\frak F}) \rightarrow \Sc(Y, {\frak F})\rightarrow 0.$$
Observe that in this exact sequence, the larger space $\Sc(X-Y, {\frak F})$ arises as a subspace, whereas the smaller space
$\Sc(Y, {\frak F})$ arises as a quotient of $\Sc(X, {\frak F})$.
Assuming that a group $G$ operates on the space $X$, preserving the closed
subspace $Y$, as well as the 
sheaf ${\frak F}$, then 
this exact sequence gives rise to an element of
$\Hom_G[\Sc(X, {\frak F}),  
\Sc(Y, {\frak F})]$, 
as well as an element of $\Ext^1_G[\Sc(Y, {\frak F}),  \Sc(X-Y, {\frak F})]$. Note that the Hom is from a larger space $\Sc(X, {\frak F})$
to a smaller space $\Sc(Y, {\frak F})$, whereas the Ext is between a smaller space and a larger space.  

Similarly, if $X_2,X_1$ 
are closed subsets of an $\ell$-space $X$ with $X_2 \subset X_1 \subset X$, and endowed with   an 
$\ell$-sheaf $\frak F$, we have exact sequences,

$$
\begin{array}{cccccccc}
0 & \rightarrow & \Sc(X-X_1, {\frak F}) & \rightarrow & \Sc(X-X_2, {\frak F}) & \rightarrow & \Sc(X_1-X_2, {\frak F}) & \rightarrow 0,\\

\\

0 &\rightarrow & \Sc(X_1-X_2, {\frak F}) & \rightarrow & \Sc(X_1, {\frak F}) & \rightarrow & \Sc(X_2, {\frak F}) & \rightarrow 0,
\end{array}
$$
which can be spliced together to give rise to the exact sequence,
$$0\rightarrow \Sc(X-X_1, {\frak F})\rightarrow \Sc(X-X_2, {\frak F}) \rightarrow  \Sc(X_1, {\frak F}) \rightarrow \Sc(X_2, {\frak F})\rightarrow 0,$$
which gives an element of $\Ext^2_G[\Sc(X_2, {\frak F}),  \Sc(X-X_1, {\frak F})]$;  so as the representation $\pi_2 = \Sc(X_2, {\frak F})
$ becomes smaller and smaller 
compared to $\pi_1 = \Sc(X-X_1, {\frak F})$ (as the space $X_2$ is `two step smaller' than $X$), 
it may be  expected to contribute to higher and higher Ext groups $\Ext_G^i[\pi_2,\pi_1]$.

Various examples around the present work suggest that homomorphisms between representations, or extensions
between them correspond to some geometric spaces (and $\ell$-sheaves on them)
as above, in particular,
a typical homomorphism  is from a 
larger space to smaller ones, whereas a typical Ext is the other way around!

Although most geometric spaces have algebraic geometric origin, that is not
necessarily the case when thinking about the Bernstein-Zelevinsky exact sequence.
For instance, one can use the action of $G$ on its Bruhat-Tits building and its various compactifications. If $X$ is the tree associated to $\PGL_2(F)$, then one knows that there is a 
compactification $\overline{X}$ of $X$ on which $\PGL_2(F)$ continues to act  with 
$\overline{X}-X 
= \PP^1(F)$, a closed subset of $\overline{X}$. 
The zero-skeleton $X^0$ of $X$ together with $\overline{X}-X $ is a compact topological 
space, call it $\overline{X}_0
$, with an action of $\PGL_2(F)$ with two orbits: $X^0$ which is the open orbit, and $\PP^1(F)$ which is the closed orbit.


An unramified character $\chi$ of $B$ gives rise to a sheaf, say $\C_\chi$, on
$\PP^1(F)$, which can be extended to a $\PGL_2(F)$-equivariant sheaf on
$\overline{X}_0$ by making it
$\mbox{\ind}_{\PGL_2({\mathcal O}_F)}^{ \PGL_2(F)} \C $ on $X^0$.
Call the extended
sheaf on $\overline{X}_0$ also as $\C_\chi$; note that the restriction
of $\C_\chi$ to $X^0$ is the constant sheaf $\C$. Thus we have an 
exact sequence,

$$0\rightarrow \Sc(X^0, \C)\rightarrow \Sc(\overline{X}_0, \C_\chi) \rightarrow \Sc(\PP^1(F), \C_\chi)\rightarrow 0.$$

Since $ \PGL_2(F)$ acts transitively on $X^0$ with stabilizer 
$\PGL_2({\mathcal O}_F)$, we have
$\Sc(X^0, \C) \cong  \mbox{\ind}_{\PGL_2({\mathcal O}_F)}^{ \PGL_2(F)} \C$,
 we  have constructed an element of  
$\Ext^1_{\PGL_2(F)}[\mbox{\Ind}_B^{\PGL_2(F)}\chi , 
\mbox{\ind}_{\PGL_2({\mathcal O}_F)}^{ \PGL_2(F)} \C 
] .$

It may be hoped that many extensions which representation theory offers will be matched by geometric action of $G$ 
on topological spaces with finitely many $G$-orbits coming either from algebraic geometric spaces, or from the Bruhat-Tits building and its compactifications; there is then also the issue of proving that geometric 
actions do give non-trivial extensions!

We end with a precise question but before that we need to make a  definition.
In what follows, groups and spaces are what are called $\ell$-groups and $\ell$-spaces in \cite{bz1} and \cite{ber}; we will also use the notion of an $\ell$-sheaf
from these references which we recall is a  sheaf, say ${\mathcal F}$,
over $X$ of vector spaces
over $\C$ with the space of compactly supported
global sections $\Sc(X,{\mathcal F})= {\mathcal F}_c(X)$, a module over $\Sc(X)$, which is
nondegenerate in the sense that $\Sc(X) \cdot {\mathcal F}_c(X)={\mathcal F}_c(X)$; the functor ${\mathcal F} \rightarrow {\mathcal F}_c(X)$ gives an equivalence between $\ell$-sheaves on $X$ and nondegenerate modules over $\Sc(X)$.
\vspace{4mm}

\noindent{\bf Definition:}
A complex representation $V$ of $G=G(F)$ is said 
to be of geometric origin if
\begin{enumerate}

\item there is a $G$-space $X_V$ with finitely
  many $G$-orbits,

\item  a $G$-equivariant sheaf ${\mathcal F}_V$ on $X_V$,

  \end{enumerate}
such that on each $G$-orbit $Y\subset X_V$ of the form $G/H_Y$, ${\mathcal F}_V|_Y$
is the equivariant sheaf associated to a finite dimensional representation $W_Y$
of $H_Y$,
and $V\cong \Sc(X_V,{\mathcal F}_V)$ (cf. \S1.16 of \cite{bz1} for the definition of the restriction of an $\ell$-sheaf to a locally closed subset such as $Y$).

\begin{example}
      Parabolic induction and Jacquet functor take representations of geometric origin to representations of geometric origin. It is expected that all
      supercuspidal representations are of geometric origin (proved for $\GL_n(F)$ and classical groups in odd residue characteristic). In understanding the class of representations of geometric origin given by the action of a group $G$ on a
      space $X$, one difficulty seems to be to glue vector bundles
      on various orbits to an $\ell$-sheaf on $X$.  
\end{example}

\begin{question}

Suppose that we have two complex
smooth representations $V_1$ and $V_2$ of $G=G(F)$ of
geometric origin with $\Ext^1_G[V_1,V_2] \not = 0$ with $V \in
\Ext^1_G[V_1,V_2]$ represented by the extension
$$0\rightarrow V_2\rightarrow V \rightarrow V_1\rightarrow 0.$$
Then is $V$ of geometric origin?
\end{question}

\begin{remark}
  A representation of {\it geometric origin} comes equipped
  with considerable additional data as in the Bernstein-Zelevinsky exact sequence,
  which may be important to refine the question above.
  \end{remark}

\end{document}